\documentclass[11pt]{amsart}
\usepackage[left=3.5cm, right=2.5cm, top=3cm, bottom=3cm]{geometry}

\usepackage{mg_macros}
\usepackage{graphicx}

\DeclareMathOperator{\sign}{sgn}

\begin{document}


\title{Optimal stopping of an $\alpha$-Brownian bridge}

\author{Maik G\"orgens}
\address{Department of Mathematics, Uppsala University}
\curraddr{P.O. Box 480, 751 06 Uppsala, Sweden}
\email{maik@math.uu.se}

\date{\today}

\begin{abstract}
  We study the problem of stopping an $\alpha$-Brownian bridge as close as possible to its global maximum. This extends earlier results found for the Brownian bridge (the case $\alpha=1$). The exact behavior for $\alpha$ close to $0$ is investigated.
\end{abstract}


\maketitle


\section{Introduction}\label{S:Intro}

We consider the stochastic differential equation
\equ[E:aBB_SDE]{ dX^{(\alpha)}_t = dW_t - \frac{\alpha X^{(\alpha)}_t}{1-t} dt, \qquad X^{(\alpha)}_0 = 0, \quad 0 \leq t < 1, }
where $\alpha \geq 0$ and $W = (W_t)_{t \in [0,1]}$ is standard Brownian motion. The unique strong solution of~\eqref{E:aBB_SDE} is given by $X^{(\alpha)}=(X^{(\alpha)}_t)_{t \in [0,1)}$ with
\equ[E:aBB_Sol]{ X^{(\alpha)}_t = \int_0^t \left( \frac{1-t}{1-s} \right)^\alpha dW_s, \qquad 0 \leq t < 1. }
If $\alpha > 0$ then $\lim_{t \rightarrow 1} X^{(\alpha)}_t = 0$ almost surely and thus $X^{(\alpha)}$ has an extension to $[0,1]$ with $X^{(\alpha)}_1 = 0$. The process $X^{(\alpha)}$ is called the $\alpha$-Brownian bridge with scaling parameter $\alpha$.

Let $\F = (\F_t)_{t \in [0,1]}$ be the natural filtration induced by $W$. We consider the optimal stopping problem
\equ[E:OS_Prob]{ V(\alpha) = \sup_{0 \leq \tau \leq 1} \E X^{(\alpha)}_\tau, }
where the supremum is taken over all $\F$-stopping times. We aim to find the values of the function $V$ as well as the stopping time $\tau^* = \tau^*(\alpha)$ for which the supremum is attained.

In the case $\alpha = 0$ we have $X^{(0)} = W$, i.e., $X^{(0)}$ is standard Brownian motion and thus a martingale with mean $0$. Hence, $\E X^{(0)}_\tau = 0$ for any stopping time $\tau$ with $0 \leq \tau \leq 1$ and so $V(0)=0$. If $\alpha$ tends to $\infty$ then $X^{(\alpha)}$ tends to the zero process and we expect that $V(\alpha)$ tends to $0$.

A possible application of our results were given in~\cite{Gor14}: as observed for example in~\cite{Ave03}, stock prices tend to end up at strike prices of heavily traded vanilla options at the time of their maturity and Brownian bridges were used there to describe this behavior. It was suggested in~\cite{Gor14} to replace the usual Brownian bridge by the $\alpha$-Brownian bridge in order to model the different behaviors of cautious ($\alpha > 1$) and incautious ($0 < \alpha < 1$) financial markets in a better way.

This work generalizes results from~\cite{Eks09} and~\cite{She69}, where the optimal stopping problem~\eqref{E:OS_Prob} was studied for the case $\alpha = 1$, i.e., for the usual Brownian bridge.

\section{The solution of the optimal stopping problem}

For $x \in \R$ and $0 \leq t \leq 1$ we introduce the value function
\equ[E:Val_Func]{ V(x, t, \alpha) = \sup_{t \leq \tau \leq 1} \E_{x,t} X^{(\alpha)}_\tau, }
where the supremum is taken over all $\F$-stopping times $\tau$ with $t \leq \tau \leq 1$ and $\E_{x,t}$ means expectation under the condition $X^{(\alpha)}_t = x$. Then the general theory of optimal stopping (see~\cite{Pes06}) yields that
\[ \tau_{t,x}^* = \inf\left\{ s \geq t: X^{(\alpha)}_s = V\left(X^{(\alpha)}_s, s, \alpha\right) \right\} \]
is optimal in~\eqref{E:Val_Func}, i.e., in order to solve the optimal stopping problem~\eqref{E:OS_Prob} we have to find $V(x, t, \alpha)$. The optimal stopping time in~\eqref{E:OS_Prob} is given by $\tau^*_{0,0}$.

From general optimal stopping theory, we expect the value function $V(x,t,\alpha)$ to solve the free boundary problem
\equ[E:PDE]{ 
\begin{cases}
  V_t(x,t,\alpha) - \frac{\alpha x}{1-t} V_x(x,t,\alpha) + \frac{1}{2} V_{xx}(x,t,\alpha) = 0, & \text{if $x < b(t,\alpha)$,} \\
  V(x,t,\alpha) = x, & \text{if $x = b(t,\alpha)$,} \\
  V_x(x,t,\alpha) = 1, & \text{if $x = b(t,\alpha)$,} \\
  V(x,t,\alpha) = 0, & \text{if $x = - \infty$,}
\end{cases}
}
where the stopping boundary $b(t,\alpha)$ is to be determined. We will solve~\eqref{E:PDE} for the different values of $\alpha > 0$. The verification that the candidate solution is the correct one can be done in exactly the same way as in~\cite{Eks09}.

With the ansatz $b(t,\alpha) = B(\alpha)\sqrt{1-t}$ and
\equ[E:Ansatz]{ V(x,t,\alpha) = \sqrt{1-t} f(x/\sqrt{1-t},\alpha), }
we obtain, with $y=x/\sqrt{1-t}$, the free boundary problem
\equ[E:ODE]{ 
\begin{cases}
  f''(y,\alpha) - (2\alpha-1)yf'(y,\alpha) - f(y,\alpha) = 0, & \text{if $y < B(\alpha)$,} \\
  f(y,\alpha) = y, & \text{if $y = B(\alpha)$,} \\
  f'(y,\alpha) = 1, & \text{if $y = B(\alpha)$,} \\
  f(y,\alpha) = 0, & \text{if $y = - \infty$.}
\end{cases}
}
Note in particular that we expect $f(\cdot, \alpha)$ to be continuously differentiable.

We introduce the confluent hypergeometric function of the first kind (see~\cite{Sla60}) by
\[ M(\gamma,\beta,z) = \sum_{n=0}^\infty \frac{\gamma^{(n)} z^n}{ \beta^{(n)} n! }, \qquad \gamma, \beta, z \in \R, \]
where the so-called Pochhammer polynomial $\gamma^{(n)}$ is defined by
\[ \gamma^{(0)} = 1, \qquad \gamma^{(n)} = (\gamma + n - 1) \gamma^{(n-1)}. \]
Moreover, by $\Gamma(\cdot)$ we denote the Gamma function. With this notation we can formulate
\begin{theorem}\label{T:1} Set
\[ \gamma(\alpha) = \alpha/(2\alpha-1) \qquad \text{and} \qquad z(y,\alpha)=y^2 (2\alpha-1)/2. \]
Then
\begin{enumerate}
  \item if $0 < \alpha < 1/2$, the differential equation in~\eqref{E:ODE} is solved by
  \[ f(y,\alpha) = C(\alpha) \left[ y M(\gamma(\alpha),3/2,z(y,\alpha) ) + \frac{\Gamma(1 - \gamma(\alpha)) M(\gamma(\alpha)-1/2,1/2,z(y,\alpha) )}{\sqrt{2(1-2\alpha)}\Gamma(3/2 - \gamma(\alpha))} \right] \]
  for $y < B(\alpha)$ and $f(y,\alpha) = y$ for $y \geq B(\alpha)$;
  \item if $\alpha = 1/2$, the free boundary problem~\eqref{E:ODE} is solved by
  \[ f(y,1/2) = \begin{cases} e^{y-1}, & \text{for $y < 1$}, \\ y & \text{for $y \geq 1$}; \end{cases} \]
  \item if $\alpha > 1/2$, the differential equation in~\eqref{E:ODE} is solved by
  \[ f(y,\alpha) = C(\alpha) \left[ y M(\gamma(\alpha),3/2,z(y,\alpha) ) + \frac{\Gamma(\gamma(\alpha) - 1/2)}{\sqrt{2(2\alpha-1)}\Gamma(\gamma(\alpha))} M(\gamma(\alpha)-1/2,1/2,z(y,\alpha) ) \right] \]
  for $y < B(\alpha)$ and $f(y,\alpha) = y$ for $y \geq B(\alpha)$.  
\end{enumerate}
  For $\alpha \neq 1/2$ the constant $B(\alpha)$ is the positive solution of
  \equ[E:B_Def]{ B(\alpha) = f(B(\alpha), \alpha) \ \big/ f'(B(\alpha), \alpha), }
  and the constant $C(\alpha)$ is determined by $B(\alpha) = f(B(\alpha), \alpha)$.
\end{theorem}


Note that the solution $B(\alpha)$ in~\eqref{E:B_Def} is unique. Otherwise we would have two different solutions of~\eqref{E:PDE}, leading to two different solutions of the optimal stopping problem~\eqref{E:Val_Func}. This would be a contradiction to the unambiguity of the definition of $V(x,t,\alpha)$.

From Theorem~\ref{T:1} we find that the solution of the partial differential equation in~\eqref{E:PDE} is given according to~\eqref{E:Ansatz}. In particular the value function and the optimal stopping time in~\eqref{E:OS_Prob} are given by
\equ[E:V0]{ V(\alpha) = V(0,0,\alpha) = 
\begin{cases} 
  0, &\text{for $\alpha = 0$,} \\
  C(\alpha) \frac{\Gamma\left(1 - \frac{\alpha}{2\alpha-1}\right)}{\sqrt{2(1-2\alpha)}\Gamma\left(\frac{3}{2} - \frac{\alpha}{2\alpha-1}\right)}, &\text{for $0 < \alpha < 1/2$,} \\
  1/e, &\text{for $\alpha = 1/2$,} \\
  C(\alpha) \frac{\Gamma\left(\frac{\alpha}{2\alpha-1} - \frac{1}{2}\right)}{\sqrt{2(2\alpha-1)}\Gamma\left(\frac{\alpha}{2\alpha-1}\right)}, &\text{for $\alpha > 1/2$.}
\end{cases} }
and
\equ[E:tau]{  \tau^* = \tau^*(\alpha) = \inf\{ t \geq 0: X^{(\alpha)}_t \geq B(\alpha) \sqrt{1-t} \}. }

\begin{remark}[The case $\alpha = 1$]
  Theorem~\ref{T:1} yields
  \[ f(y, 1) = C \left(y M(1,3/2,y^2/2) + \sqrt{\pi/2} M(1/2,1/2,y^2/2) \right). \]
  Since $M(\gamma,\gamma,z) = e^z$ and
  \[ M(1, 3/2, z) = \frac{e^z}{\sqrt{z}} \int_0^{\sqrt{z}} e^{-s^2} ds, \]
  we obtain
  \begin{align*}
    f(y, 1) &= C \left[ \sign(y) e^{y^2/2} \int_0^{|y|} e^{-s^2/2} ds  + \sqrt{\pi/2} e^{y^2/2} \right] \\
      &= C \sqrt{\pi/2} e^{y^2/2} \left[ \sign(y) \sqrt{\frac{2}{\pi}} \int_0^{|y|} e^{-t^2/2} dt + 1 \right] \\
      &= C \sqrt{\pi/2} e^{y^2/2} \left[ \sign(y) (2 \Phi(|y|) - 1) + 1 \right] \\
      &= C \sqrt{2\pi} e^{y^2/2} \Phi(y)
  \end{align*}
  for all $y \in \R$, where
  \[ \Phi(y) = \frac{1}{\sqrt{2\pi}} \int_{-\infty}^y e^{-s^2/2} ds. \]
  From this we recover the result found in~\cite{Eks09} for $\alpha = 1$.
\end{remark}

The remaining part of this section is devoted to the proof of Theorem~\ref{T:1}.

\subsection{The case $\alpha = 1/2$} In this case the differential equation in~\eqref{E:ODE} reduces to
\[ f''(y,1/2) - f(y,1/2) = 0 \qquad \text{if $y < B(1/2)$,} \]
which has the general solution
\[ f(y,1/2) = C e^y + D e^{-y}. \]
The requirement $f(y, 1/2) \rightarrow 0$ as $y \rightarrow - \infty$ yields $D = 0$ and the conditions $f(y,1/2) = y$ and $f'(y,1/2) = 1$ for $y = B(1/2)$ yield $B(1/2)=1$ and $C = e^{-1}$.


\subsection{The case $\alpha \neq 1/2$} The function $f(y, \alpha)$ may be written as the sum of an odd function $f_1(y, \alpha)$ and an even function $f_2(y, \alpha)$. The ansatz $f_1(y, \alpha) = y g_1(y^2 (2\alpha-1)/2)$ and $z=z(y,\alpha)=y^2 (2\alpha-1)/2$ turns the differential equation
\[ f''(y) - (2\alpha-1)yf'(y) - f(y) = 0 \]
into
\[ z g_1''(z) + (3/2 - z) g_1'(z) - \frac{\alpha}{2\alpha-1}g_1(z) = 0, \]
which is Kummer's differential equation
\equ[E:Kummer]{ z g''(z) + (\beta - z) g'(z) - \gamma g(z) = 0 }
with parameters $\beta=3/2$ and $\gamma = \gamma(\alpha) = \alpha/(2\alpha-1)$. The ansatz $f_2(y, \alpha) = |y| g_2(y^2 (2\alpha-1)/2)$ yields Kummer's differential equation with the same parameters $\beta$ and $\gamma$ for $g_2$.

One solution of~\eqref{E:Kummer} is $M(\gamma,\beta,z)$. The asymptotic behavior of $M$ is (see formula~(4.1.7) and formula~(4.1.8) in~\cite{Sla60})
\begin{align}
  M(\gamma,\beta,z) &= \frac{\Gamma(\beta)}{\Gamma(\gamma)} e^z z^{\gamma-\beta} \left( 1 + \mathcal{O} \left( z^{-1} \right)\right), \qquad \text{as $z \rightarrow \infty$,} \quad \text{and} \label{E:Asympt_M+} \\
  M(\gamma,\beta,z) &= \frac{\Gamma(\beta)}{\Gamma(\beta-\gamma)} (-z)^{-\gamma} \left( 1 + \mathcal{O} \left( (-z)^{-1} \right)\right), \qquad \text{as $z \rightarrow -\infty$,} \label{E:Asympt_M-}
\end{align}
and the derivative of $M$ with respect to $z$ is
\equ[E:Deriv_M]{ M'(\gamma,\beta,z) = \gamma/\beta M(\gamma+1,\beta+1,z). }

A second solution of~\eqref{E:Kummer} is
\[ N(\gamma,\beta,z) = z^{1-\beta} M(\gamma-\beta+1, 2-\beta, z). \]
Since the values of $N$ are complex for negative $z$ and $\beta=3/2$ (as in our case), and since the asymptotic behavior of $M$ as $|z| \rightarrow \infty$ depends on the sign of $z$ we distinguish between positive $z$ (i.e., $\alpha > 1/2$) and negative $z$ (i.e., $\alpha < 1/2$).

\subsubsection{The case $\alpha > 1/2$}

We have $\gamma(\alpha) > 1/2$ and $z(y,\alpha) > 0$. Setting
\begin{align*}
  U(\gamma,\beta,z) &= \frac{\Gamma(1-\beta)}{\Gamma(\gamma-\beta+1)}M(\gamma,\beta,z)+\frac{\Gamma(\beta-1)}{\Gamma(\gamma)}N(\gamma,\beta,z) \\[4mm]
           &= \frac{\Gamma(1-\beta)}{\Gamma(\gamma-\beta+1)}M(\gamma,\beta,z)+\frac{\Gamma(\beta-1)}{\Gamma(\gamma)}z^{1-\beta}M(\gamma-\beta+1,2-\beta,z),
\end{align*}
we obtain by~\eqref{E:Asympt_M+} a solution of~\eqref{E:Kummer} with
\equ[E:Asympt_U+]{ U(\gamma,\beta,z) - z^{-\gamma} \longrightarrow 0, \qquad \text{as $z \rightarrow \infty$.} }
The function $U$ is called the confluent hypergeometric function of the second kind (again, see~\cite{Sla60} for more details). Since $M(\gamma,\beta,0) = 1$ for all $\beta, \gamma \in \R$, we have, as $z \rightarrow 0$,
\begin{align}
  U(\gamma,\beta,z) &- \frac{\Gamma(\beta-1)}{\Gamma(\gamma)}z^{1-\beta} - \frac{\Gamma(1-\beta)}{\Gamma(\gamma-\beta+1)}\longrightarrow 0, &\text{for $1 < \beta < 2$, and} \label{E:Asympt_U0} \\[4mm]
  U(\gamma,\beta,z) &- \frac{\Gamma(\beta-1)}{\Gamma(\gamma)}z^{1-\beta} - \mathcal{O}\left( z^{\beta-2} \right)\longrightarrow 0, &\text{for $\beta > 2$.} \notag
\end{align}
Moreover, 
\equ[E:Deriv_U]{ U'(\gamma,\beta,z) = -\gamma U(\gamma+1, \beta+1, z). }

We obtain the general solution for $f(\cdot, \alpha)$ as
\begin{align*}
  f(y, \alpha) &= C_1 y g_1(z(y,\alpha)) + C_2 |y| g_2(z(y,\alpha)) \\[3mm]
    &= C_1 y M(\gamma(\alpha),3/2,z(y,\alpha)) + C_2 y U(\gamma(\alpha),3/2,z(y,\alpha)) \\[2mm]
    &\qquad + C_3 |y| M(\gamma(\alpha),3/2,z(y,\alpha)) + C_4 |y| U(\gamma(\alpha),3/2,z(y,\alpha)) \\[3mm]
    &= (C_1 y + C_3 |y|) M(\gamma(\alpha),3/2,z(y,\alpha)) + (C_2 y + C_4 |y|) U(\gamma(\alpha),3/2,z(y,\alpha)).
\end{align*}
The asymptotic behavior of $M(\gamma,\beta,z)$ and $U(\gamma,\beta,z)$ described in~\eqref{E:Asympt_M+} and~\eqref{E:Asympt_U+}, together with the requirement $f(y, \alpha) \rightarrow 0$ as $y \rightarrow - \infty$, yields $C_1 = C_3$ and thus
\[ f(y, \alpha) = C_1 (y + |y|) M(\gamma(\alpha),3/2,z(y,\alpha)) + (C_2 y + C_4 |y|) U(\gamma(\alpha),3/2,z(y,\alpha)). \]
Next, since we require $f(\cdot, \alpha)$ to be continuous at $0$, we need
\begin{align*}
  \lim_{y \nearrow 0}(C_2 y - C_4 y) U(\gamma(\alpha),3/2,z(y,\alpha)) &= (C_4 - C_2 ) \lim_{y \searrow 0} y U(\gamma(\alpha),3/2,z(y,\alpha)) \\
    &= (C_4 - C_2) \frac{\sqrt{2\pi}}{\sqrt{2\alpha-1} \Gamma( \gamma(\alpha))}
\end{align*}
to be equal to
\[ \lim_{y \searrow 0}(C_2 y + C_4 y) U(\gamma(\alpha),3/2,z(y,\alpha)) = (C_2 + C_4) \frac{\sqrt{2\pi}}{\sqrt{2\alpha-1} \Gamma(\gamma(\alpha))}, \]
where we used~\eqref{E:Asympt_U0} to calculate the limits. Thus, $C_2 = 0$ and we get
\[ f(y, \alpha) = C_1 (y + |y|) M(\gamma(\alpha),3/2,z(y,\alpha)) + C_4 |y| U(\gamma(\alpha),3/2,z(y,\alpha) ). \]
We also require $f'(\cdot, \alpha)$ to be continuous at $0$. Using~\eqref{E:Deriv_M}, \eqref{E:Asympt_U0}, and~\eqref{E:Deriv_U} we thus need
\begin{align*}
  \lim_{y \nearrow 0} f'(y,\alpha) &= \lim_{y \nearrow 0} \left[ - C_4 U(\gamma(\alpha),3/2,z(y,\alpha)) + C_4 \alpha y^2 U(\gamma(\alpha) + 1,5/2,z(y,\alpha) ) \right] \\
  &= C_4 D(\alpha)
\end{align*}
to be equal to
\begin{align*}
  \lim_{y \searrow 0} f'(y,\alpha) &= \lim_{y \searrow 0} \Big[ 2 C_1 M(\gamma(\alpha),3/2,z(y,\alpha) ) + 4/3 C_1 \alpha y^2 M(\gamma(\alpha) + 1,5/2,z(y,\alpha) ) \\
  &\qquad \qquad + C_4 U(\gamma(\alpha),3/2,z(y,\alpha) ) - C_4 \alpha y^2 U(\gamma(\alpha) + 1,5/2,z(y,\alpha) ) \Big] \\
  &= 2 C_1 - C_4 D(\alpha),
\end{align*}
where
\[ D(\alpha) = \lim_{y \searrow 0} \left[ \alpha y^2 U(\gamma(\alpha) + 1,5/2,z(y,\alpha) ) -  U(\gamma(\alpha),3/2,z(y,\alpha) ) \right]. \]
This necessitates that $C_4 D(\alpha) = 2 C_1 - C_4 D(\alpha)$ or equivalently $C_4 = C_1 /D(\alpha)$. Using~\eqref{E:Asympt_U0} we obtain
\begin{align*}
  D(\alpha) &= \lim_{y \searrow 0} \Bigg[ \alpha y^2 \frac{\Gamma(3/2)}{\Gamma\left(\gamma(\alpha)+1\right)} (z(y,\alpha))^{-3/2} + \alpha y^2 \mathcal{O} \left( (z(y,\alpha))^{1/2} \right) \\
    &\qquad \qquad \qquad - \frac{\Gamma(1/2)}{\Gamma\left(\gamma(\alpha)\right)} (z(y,\alpha))^{-1/2} - \frac{\Gamma(-1/2)}{\Gamma(\gamma(\alpha) - 1/2 )} \Bigg] \\
  &= - \frac{\Gamma(-1/2)}{\Gamma(\gamma(\alpha) - 1/2 )} = \frac{2 \sqrt{\pi}}{\Gamma(\gamma(\alpha) - 1/2)}
\end{align*}
and thus
\[ C_4 = C_1 \Gamma(\gamma(\alpha)-1/2) / (2\sqrt{\pi}). \]
This finally yields
\begin{align*} 
  f(y,\alpha) &= C(\alpha) \left[ (y + |y|) M(\gamma(\alpha),3/2,z(y,\alpha) ) + \frac{|y| \Gamma(\gamma(\alpha)-1/2)}{2\sqrt{\pi}} U(\gamma(\alpha),3/2,z(y,\alpha) ) \right] \\[2mm]
 &= C(\alpha) \left[ y M(\gamma(\alpha),3/2,z(y,\alpha) ) + \frac{\Gamma(\gamma(\alpha) - 1/2)}{\sqrt{2(2\alpha-1)}\Gamma(\gamma(\alpha))} M(\gamma(\alpha)-1/2,1/2,z(y,\alpha) ) \right].
\end{align*}

From $B(\alpha) = f(B(\alpha), \alpha)$ and $1 = f'(B(\alpha), \alpha)$ we obtain the following equation in $B(\alpha)$, which is independent of $C(\alpha)$:
\equ[_E:Neu1]{ B(\alpha) = f(B(\alpha), \alpha) \ \big/ f'(B(\alpha), \alpha). }
In order to see that~\eqref{_E:Neu1} admits a positive solution, note that the function $h(y, \alpha) = f(y, \alpha) - y f'(y, \alpha)$ is continuous with
\[ h(0, \alpha) = \frac{\Gamma(\gamma(\alpha) - 1/2)}{\sqrt{2(2\alpha-1)}\Gamma(\gamma(\alpha))} > 0 \qquad \text{and} \qquad \lim_{y\rightarrow\infty} f(y, \alpha) = - \infty.  \]
Finally, $C(\alpha)$ is obtained via the relation $B(\alpha) = f(B(\alpha), \alpha)$.

\subsubsection{The case $0 < \alpha < 1/2$} We have $\gamma(\alpha) < 0$ and $z(y, \alpha) < 0$ and with $\beta=3/2$ the value of $N(\gamma,\beta,z)$ is an imaginary number. We set
\begin{align}
  W(\gamma,3/2,z) &= \frac{\Gamma(3/2-\gamma)}{\Gamma(3/2)}M(\gamma,3/2,z) - i \frac{\Gamma(1-\gamma)}{\Gamma(1/2)}N(\gamma,3/2,z) \notag \\[4mm]
           &= \frac{\Gamma(3/2-\gamma)}{\Gamma(3/2)} M(\gamma,3/2,z) - \frac{\Gamma(1-\gamma)}{\Gamma(2-3/2)} (-z)^{1-3/2} M(\gamma-3/2+1,2-3/2,z), \label{E:W_prel}
\end{align}
and use~\eqref{E:W_prel} as a template to define $W(\gamma,\beta,z)$ for all $\beta$ in the following way:
\[ W(\gamma,\beta,z) = \frac{\Gamma(\beta-\gamma)}{\Gamma(\beta)} M(\gamma,\beta,z) - \frac{\Gamma(1-\gamma)}{\Gamma(2-\beta)} (-z)^{1-\beta} M(\gamma-\beta+1,2-\beta,z). \]
In this way we obtain, by~\eqref{E:Asympt_M-}, a solution of~\eqref{E:Kummer} with
\[ W(\gamma,\beta,z) \rightarrow 0, \qquad \text{as $z \rightarrow -\infty$,} \]
and
\[ W'(\gamma,\beta,z) = \gamma W(\gamma+1, \beta+1 ,z). \]

The following calculations are very similar to the ones in the previous section. Therefore, we skip some details. As before, we find
\[ f(y, \alpha) = (C_1 y + C_3 |y|) M(\gamma(\alpha),3/2,z(y,\alpha) ) + (C_2 y + C_4 |y|) W(\gamma(\alpha),3/2,z(y,\alpha) ). \]
From $f(y) \rightarrow 0$ as $y \rightarrow - \infty$ and the requirement that $f(\cdot, \alpha)$ is continuous at $0$ we get $C_1 = C_3$ and $C_2 = 0$, and so
\[ f(y, \alpha) = C_1 (y + |y|) M(\gamma(\alpha),3/2,z(y,\alpha) ) + C_4 |y| W(\gamma(\alpha),3/2,z(y,\alpha) ). \]
We also require $f'(\cdot, \alpha)$ to be continuous at $0$, i.e., we require
\begin{align*}
  \lim_{y \nearrow 0} f'(y,\alpha) &= \lim_{y \nearrow 0} \left[ - C_4 W(\gamma(\alpha),3/2,z(y,\alpha) ) - C_4 \alpha y^2 W(\gamma(\alpha) + 1,5/2,z(y,\alpha) ) \right] \\
  &= - C_4 E(\alpha)
\end{align*}
to be equal to
\begin{align*}
  \lim_{y \searrow 0} f'(y,\alpha) &= \lim_{y \searrow 0} \Big[ 2 C_1 M(\gamma(\alpha),3/2,z(y,\alpha) ) + 4/3 C_1 \alpha y^2 M(\gamma(\alpha) + 1,5/2,z(y,\alpha) ) \\
  &\qquad + C_4 W(\gamma(\alpha),3/2,z(y,\alpha) ) + C_4 \alpha y^2 W(\gamma(\alpha) + 1,5/2,z(y,\alpha) ) \Big] \\[2mm]
  &= 2 C_1 + C_4 E(\alpha),
\end{align*}
where
\[ E(\alpha) = \lim_{y \searrow 0} \left[ \alpha y^2 W(\gamma(\alpha) + 1,5/2,z(y,\alpha) ) +  W(\gamma(\alpha),3/2,z(y,\alpha) ) \right] = \frac{\Gamma(3/2-\gamma(\alpha))}{\Gamma(3/2)}. \]
Hence we need $-C_4 E(\alpha) = 2 C_1 + C_4 E(\alpha)$ or equivalently
\[ C_4 = - C_1 / E(\alpha) = - C_1 \Gamma(3/2) \big/ \Gamma(3/2-\gamma(\alpha)) \]
which leads to
\begin{align*} 
  f(y,\alpha) &= C(\alpha) \left[ (y + |y|) M(\gamma(\alpha),3/2,z(y,\alpha) ) - \frac{|y| \sqrt{\pi}}{2\Gamma(3/2-\gamma(\alpha))} W(\gamma(\alpha),3/2,z(y,\alpha) ) \right] \\[2mm]
  &= C(\alpha) \left[ y M(\gamma(\alpha),3/2,z(y,\alpha) ) + \frac{\Gamma(1 - \gamma(\alpha)) M(\gamma(\alpha)-1/2,1/2,z(y,\alpha) )}{\sqrt{2(1-2\alpha)}\Gamma(3/2 - \gamma(\alpha))} \right].
\end{align*}
The constants $B(\alpha)$ and $C(\alpha)$ are obtained as before.


\section{Asymptotic behavior}



Consider $X: [0, \infty) \times [0,1] \rightarrow \R$ defined by $X(\alpha, t) = X^{(\alpha)}_t$ with $X^{(\alpha)}_t$ represented as in~\eqref{E:aBB_Sol}. This is a Gaussian random field which is almost surely not continuous at the point $(0, 1)$, since $\Prob(X^{(0)}_1 = 0) = 0$, whereas $\Prob(X^{(\alpha)}_1 = 0) = 1$ for all $\alpha > 0$. We will study how this discontinuity affects the continuity of $V(\alpha)$ at $\alpha = 0$. Note that, as mentioned in Section~\ref{S:Intro}, we have $V(0)=0$.

\begin{theorem}\label{T:LIM}
  The limiting behavior of $B(\alpha)$, $C(\alpha)$, and $V(\alpha)$ for $\alpha$ close to zero is
  \begin{enumerate}
    \item $\lim_{\alpha \searrow 0} B(\alpha) = \infty$,
    \item $\lim_{\alpha \searrow 0} C(\alpha) = 1/2$,
    \item $\lim_{\alpha \searrow 0} V(\alpha) = 1/\sqrt{2\pi}$.
  \end{enumerate}
\end{theorem}

\begin{proof}
In order to prove (i), assume that there is an $\eps > 0$ and a constant $K > 0$ such that $B(\alpha) < K$ for all $0 < \alpha < \eps$. By Theorem~\ref{T:1}, the constant $B(\alpha)$ is given by the equation $B(\alpha) = f(B(\alpha), \alpha) / f'(B(\alpha), \alpha)$, where
\equ[E:f_again]{ f(y, \alpha) = C(\alpha) \left[ y M(\gamma(\alpha),3/2,z(y,\alpha) ) + \frac{\Gamma(1 - \gamma(\alpha)) M(\gamma(\alpha)-1/2,1/2,z(y,\alpha) )}{\sqrt{2(1-2\alpha)}\Gamma(3/2 - \gamma(\alpha))} \right]. }
The derivative of $f(y, \alpha)$ with respect to $y$ is
\begin{align*}
  f'(y, \alpha) &= C(\alpha) \Bigg[ M(\gamma(\alpha),3/2,z(y,\alpha)) + \frac{2 \alpha y^2 }{3} M(\gamma(\alpha) + 1,5/2,z(y,\alpha)) \\
    &\qquad \qquad + \frac{y \Gamma(1 - \gamma(\alpha)) M(\gamma(\alpha)+1/2,3/2,z(y,\alpha) ) }{\sqrt{2(1-2\alpha)}\Gamma(3/2 - \gamma(\alpha))}  \Bigg].
\end{align*}
From the requirement $B(\alpha) f'(B(\alpha), \alpha) = f(B(\alpha), \alpha)$ we get
\[ \frac{\sqrt{2(1-2\alpha)} \Gamma(3/2 - \gamma(\alpha)) B(\alpha) f'(B(\alpha), \alpha)}{C(\alpha) \Gamma(1-\gamma(\alpha))} = \frac{\sqrt{2(1-2\alpha)} \Gamma(3/2 - \gamma(\alpha)) f(B(\alpha), \alpha)}{C(\alpha) \Gamma(1-\gamma(\alpha))}, \]
which, after plugging in $f(y, \alpha)$ and $f'(y, \alpha)$ and some reordering, gives
\begin{align}
0 = &- M(\gamma(\alpha) - 1/2, 1/2, z(B(\alpha), \alpha)) + B(\alpha)^2 M(\gamma(\alpha) + 1/2, 3/2, z(B(\alpha), \alpha)) \label{E:NewInequ} \\
  &+ \alpha \frac{2 \sqrt{2(1-2\alpha)} \Gamma(3/2 - \gamma(\alpha)) B(\alpha)^3}{3 \Gamma(1-\gamma(\alpha))} M(\gamma(\alpha) + 1, 5/2, z(\gamma(\alpha))). \notag
\end{align}
From the assumption that $B(\alpha) < K$ for all $0 < \alpha < \eps$, the continuity of $M$ in all three parameters, and the fact that $\gamma(\alpha) \rightarrow 0$ and $z(B(\alpha), \alpha) + B(\alpha)^2/2 \rightarrow 0$ as $\alpha \rightarrow 0$, we get for small enough $\eps > 0$
\[ M(\gamma(\alpha) - 1/2, 1/2, z(B(\alpha), \alpha)) \geq M(-1/2, 1/2, - B(\alpha)^2/2) - e^{-K^2/2} / 3 \]
and
\[ B(\alpha)^2 M(\gamma(\alpha) + 1/2, 3/2, z(B(\alpha), \alpha)) \leq e^{-K^2/2} / 3 + B(\alpha)^2 M(1/2, 3/2, - B(\alpha)^2/2), \]
as well as
\[ \alpha \frac{2 \sqrt{2(1-2\alpha)} \Gamma(3/2 - \gamma(\alpha)) B(\alpha)^3}{3 \Gamma(1-\gamma(\alpha))} M(\gamma(\alpha) + 1, 5/2, z(\gamma(\alpha))) \leq e^{-K^2/2} / 3. \]
Inserting those estimates into~\eqref{E:NewInequ} yields
\begin{align}
  0 &\leq - M(-1/2, 1/2, - B(\alpha)^2/2) + e^{-K^2/2} / 3 \notag \\[2mm]
    &\qquad + e^{-K^2/2} / 3 + B(\alpha)^2 M(1/2, 3/2, - B(\alpha)^2/2) + e^{-K^2/2} / 3 \notag \\[3mm]
    &= - M(-1/2, 1/2, - B(\alpha)^2/2) - 2 \frac{-B(\alpha)^2}{2} M(1/2, 3/2, - B(\alpha)^2/2) + e^{-K^2/2}. \label{E:NewInequ2}
\end{align}
For $\beta \neq 0$, the function $M$ fulfills the following recurrence relation (see formula (2.2.4) in~\cite{Sla60})
\equ[E:M_Rec]{ M(\gamma, \beta, z) - M(\gamma-1, \beta, z) - z/\beta M(\gamma, \beta+1, z) = 0. }
Applying this (with $\beta = \gamma = 1/2$ and $z = - B(\alpha)^2/2$) to~\eqref{E:NewInequ2} yields
\[ 0 \leq - M(1/2, 1/2, - B(\alpha)^2/2) + e^{-K^2/2} = - e^{- B(\alpha)^2/2} + e^{-K^2/2}. \]
This is a contradiction to the assumption that $B(\alpha) < K$ for all $0 < \alpha < \eps$. Hence, (i) is proven.

By Theorem~\ref{T:1}, the constant $C(\alpha)$ is determined by the equation $B(\alpha) = f(B(\alpha), \alpha)$, where $f(B(\alpha), \alpha)$ is given in~\eqref{E:f_again}. It follows that
\begin{align*}
  \lim_{\alpha \searrow 0} \frac{1}{C(\alpha)} &= \lim_{\alpha \searrow 0} \left[ M(\gamma(\alpha),3/2,z(B(\alpha),\alpha) ) + \frac{\Gamma(1 - \gamma(\alpha)) M(\gamma(\alpha)-1/2,1/2,z(B(\alpha),\alpha) )}{B(\alpha) \sqrt{2(1-2\alpha)}\Gamma(3/2 - \gamma(\alpha))} \right] \\[2mm]
    &=\lim_{\alpha \searrow 0} \left[ M(\gamma(\alpha),3/2,-B(\alpha)^2/2 ) + \frac{ M(-1/2,1/2,-B(\alpha)^2/2 )}{B(\alpha) \sqrt{2}\Gamma(3/2)} \right].
\end{align*}
We apply~\eqref{E:M_Rec} once again and obtain
\begin{align*}
  \lim_{\alpha \searrow 0} \frac{1}{C(\alpha)} &= \lim_{\alpha \searrow 0} \Bigg[M(\gamma(\alpha) + 1, 3/2, -B(\alpha)^2/2) + B(\alpha)^2/3 M(\gamma(\alpha) + 1,5/2,-B(\alpha)^2/2 ) ) \\
  &\qquad\qquad\qquad + \frac{ M(-1/2,1/2,-B(\alpha)^2/2 )}{B(\alpha) \sqrt{2}\Gamma(3/2)} \Bigg] \\
  &= \lim_{\alpha \searrow 0} \Bigg[M(1, 3/2, -B(\alpha)^2/2) + B(\alpha)^2/3 M(1,5/2,-B(\alpha)^2/2 ) ) \\
  &\qquad\qquad\qquad + \frac{ M(-1/2,1/2,-B(\alpha)^2/2 )}{B(\alpha) \sqrt{2}\Gamma(3/2)} \Bigg].
\end{align*}
Since $B(\alpha) \rightarrow \infty$ as $\alpha \searrow 0$ we get by~\eqref{E:Asympt_M-}
\begin{align*}
  \lim_{\alpha \searrow 0} \frac{1}{C(\alpha)} &= \lim_{\alpha \searrow 0} \Bigg[ \frac{\Gamma(3/2)}{\Gamma(1/2)} \left( \frac{B(\alpha)^2}{2} \right)^{-1} + \frac{B(\alpha)^2 \Gamma(5/2)}{3 \Gamma(3/2)} \left( \frac{B(\alpha)^2}{2} \right)^{-1} \\
    &\qquad\qquad\qquad + \frac{\Gamma(1/2)}{\sqrt{2} \Gamma(1) B(\alpha) \Gamma(3/2) } \left( \frac{B(\alpha)^2}{2} \right)^{1/2} \Bigg] \\
  &= 2,
\end{align*}
which finishes the proof of~(ii).

The statement~(iii) follows immediately from~(ii) and~\eqref{E:V0}. That is,
\[ \lim_{\alpha \searrow 0} V(\alpha) = \lim_{\alpha \searrow 0} C(\alpha) \frac{\Gamma(1 - \gamma(\alpha))}{\sqrt{2(1-2\alpha)}\Gamma(3/2 - \gamma(\alpha))} = \frac{1}{2} \frac{\Gamma(1)}{\sqrt{2} \Gamma(3/2)} = \frac{1}{\sqrt{2\pi}}. \qedhere \]
\end{proof}

\begin{remark}
  We conclude this section with a heuristic argument for the fact $\lim_{\alpha \searrow 0} V(\alpha) = 1/\sqrt{2 \pi}$: For very small $\alpha > 0$, the process $X^{(\alpha)}$ behaves roughly like standard Brownian motion $W = (W_t)_{t \in [0,1]}$, but jumps to $0$ at time $1$. This suggests that
  \[ \lim_{\alpha \searrow 0} V(\alpha) = \E \ G\left(W_1\right), \qquad \text{where} \qquad G(x) = \begin{cases} x, &\text{if $x \geq 0$,} \\ 0, &\text{if $x < 0$.} \end{cases} \]
  Since $W_1$ is a standard normal (and thus symmetric) random variable, we get
  \[ \lim_{\alpha \searrow 0} V(\alpha) = \frac{1}{2} \ \E |W_1| = \frac{1}{2} \ \frac{\sqrt{2}}{\sqrt{\pi}} = \frac{1}{\sqrt{2\pi}}. \]
\end{remark}


\section{Numerical results and discussion}

Based on Theorem~\ref{T:1} and~\eqref{E:V0} we have computed the constants $B(\alpha)$, $C(\alpha)$, and $V(\alpha)$ for different values of $\alpha$ numerically. A plot of $B(\alpha)$ for $0 < \alpha \leq 10$ is given in Figure~\ref{Fig:B}. It indicates that $\lim_{\alpha \searrow 0} B(\alpha) = \infty$ as proven in Theorem~\ref{T:LIM}(i). This implies that the stopping boundary $b(t,\alpha) = B(\alpha) \sqrt{1-t}$ fulfills $\lim_{\alpha \searrow 0} b(t, \alpha) = \infty$ for all $0 \leq t < 1$. On the other hand, $X^{(\alpha)}$ tends to a Brownian motion that jumps to $0$ at time $1$ and thus we expect that $\lim_{\alpha \searrow 0} \tau^*(\alpha) = 1$ almost surely as $\alpha \searrow 0$. 

A plot of $V(\alpha)$ for $0 < \alpha \leq 10$ is given in Figure~\ref{Fig:Va}. In Theorem~\ref{T:LIM}(iii) we have shown that $\lim_{\alpha \searrow 0} V(\alpha) = 1/\sqrt{2\pi} \approx 0.4$, which can be seen in the plot. However, as mentioned in Section~\ref{S:Intro}, we have $V(0) = 0$ and so $V(\alpha)$ is not continuous in $\alpha = 0$. We also see $V(1) \approx 0.37$ as computed in~\cite{Eks09}.

A plot of $V(\alpha)$ for $0.3 \leq \alpha \leq 1.3$ is given in Figure~\ref{Fig:Vb}. As can be seen, $V(\alpha)$ has a local minimum at $\alpha \approx 0.5$ (we conjecture that the local minimum is exactly at $\alpha = 0.5$ but we have not been able to prove this) and a local maximum at $\alpha \approx 0.98$. The non-monotonicity of $V(\alpha)$ can be explained in the following way: increasing $\alpha$ has a decreasing and an increasing effect on the supremum in~\eqref{E:OS_Prob}, since the drift term in~\eqref{E:aBB_SDE} decreases when $X^{(\alpha)}_t$ is positive and increases when $X^{(\alpha)}_t$ is negative. Figure~\ref{Fig:Vb} shows that the first effect dominates the second one for most but not all values of $\alpha$.


\bibliographystyle{plain}
\bibliography{bibl}


\newpage

\begin{figure}
  \caption{The values of $B(\alpha)$ for $0 < \alpha \leq 10$.}\label{Fig:B}
  \includegraphics[width=0.8\textwidth]{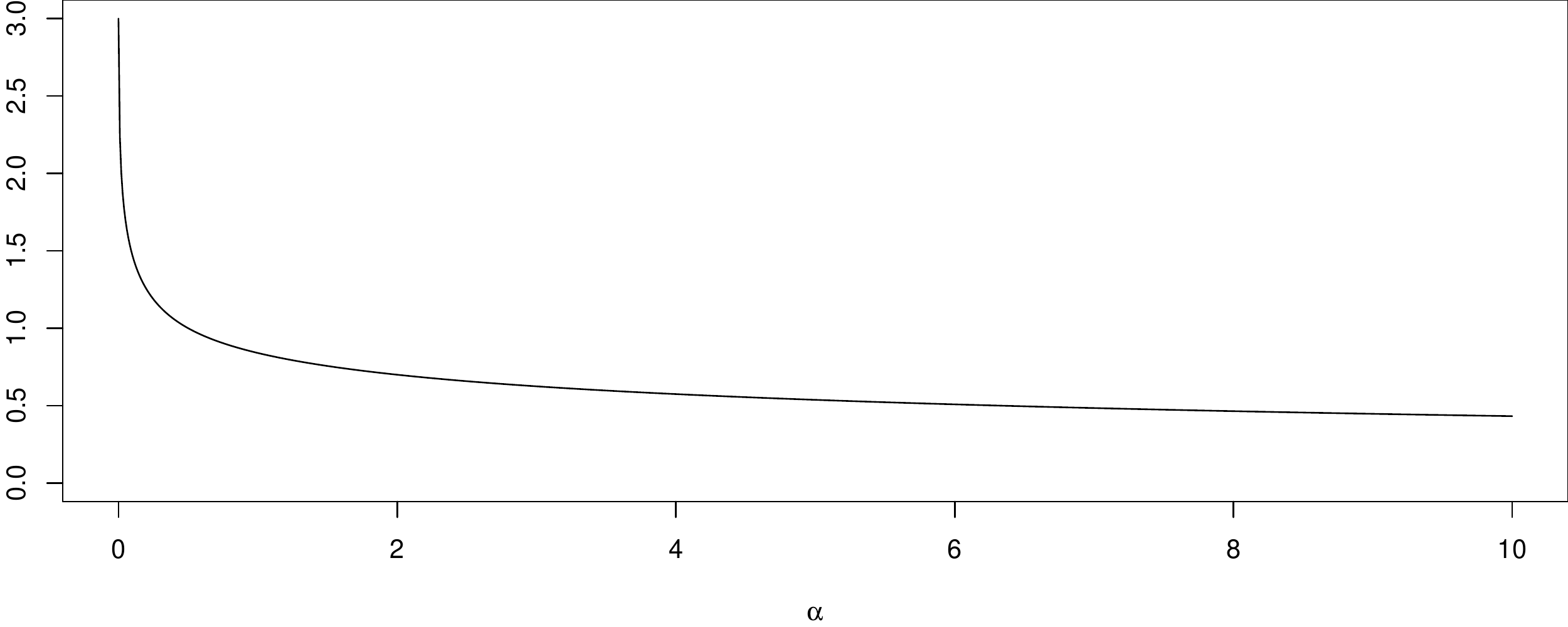}
\end{figure}

\begin{figure}
  \caption{The values of $V(\alpha)$ for $0 < \alpha \leq 10$.}\label{Fig:Va}
  \includegraphics[width=0.8\textwidth]{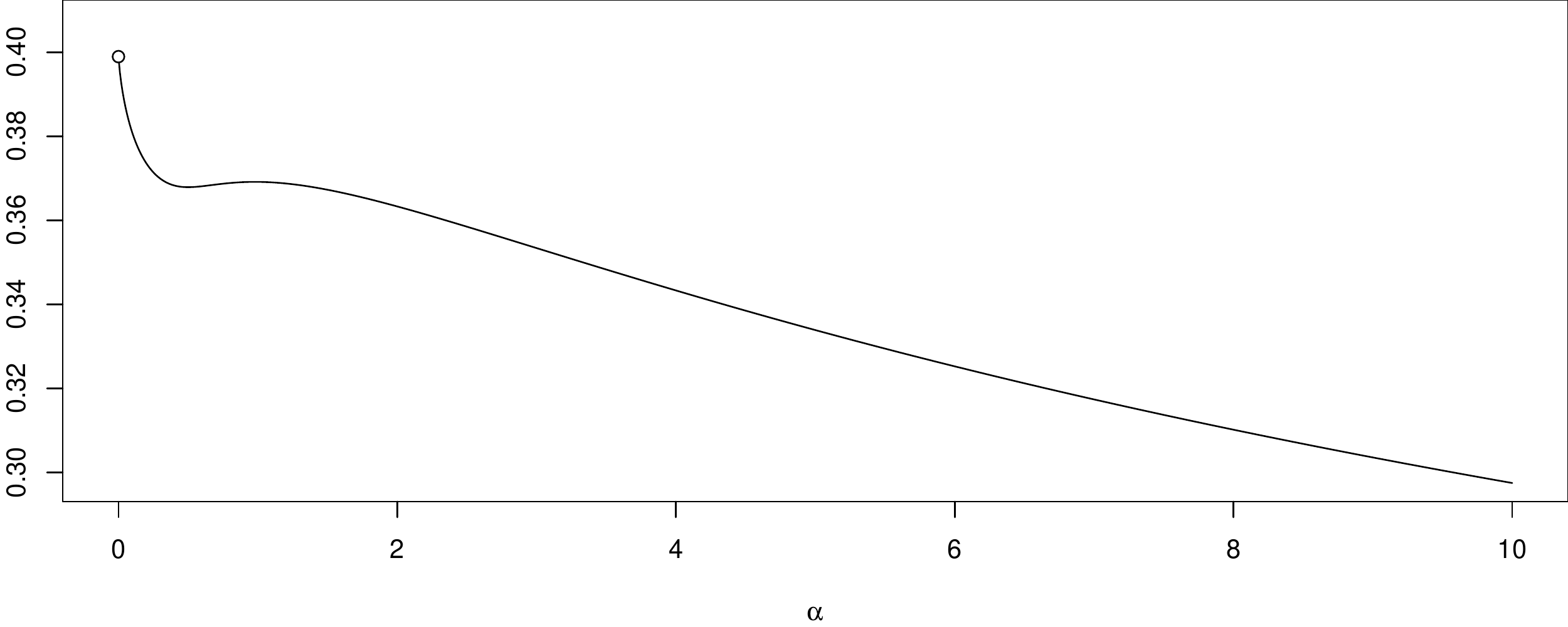}
\end{figure}

\begin{figure}
  \caption{The values of $V(\alpha)$ for $0.3 \leq \alpha \leq 1.3$.}\label{Fig:Vb}
  \includegraphics[width=0.8\textwidth]{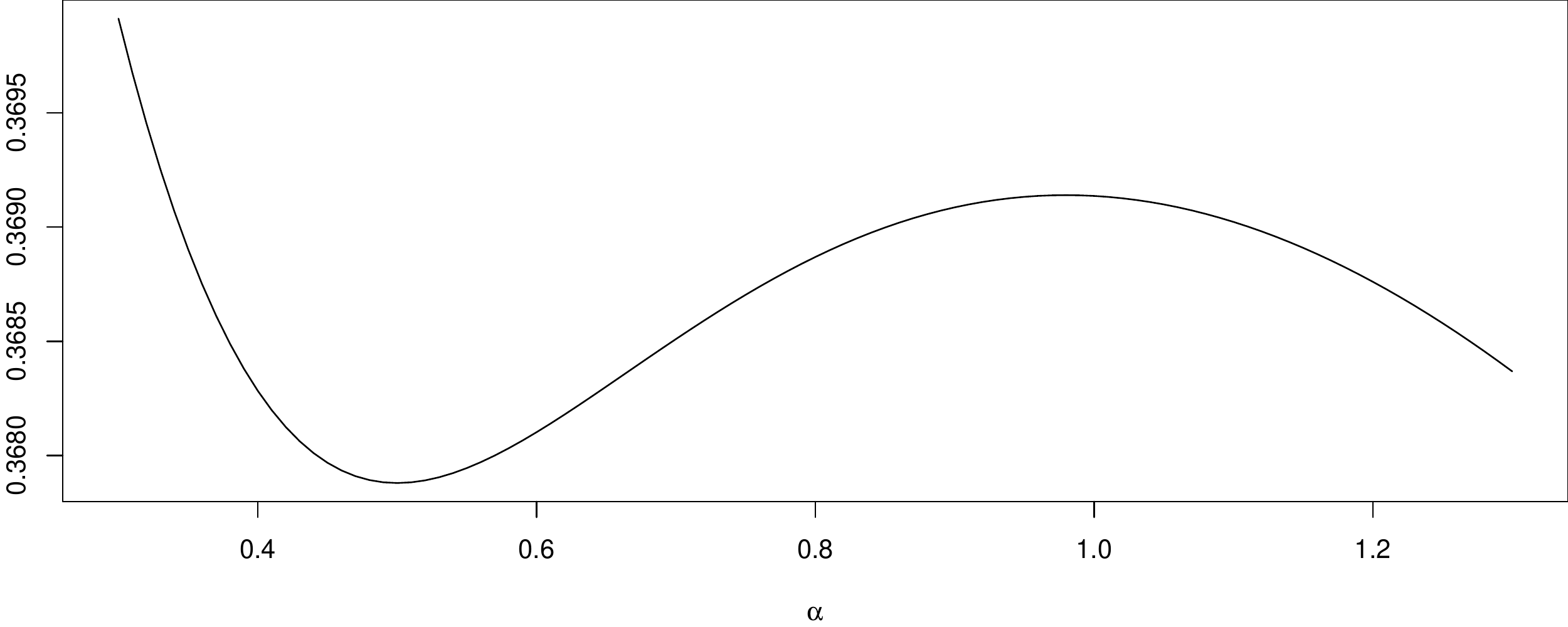}
\end{figure}


\end{document}